\documentclass[12pt]{amsart}
\usepackage{amssymb,amsmath,amsthm,latexsym,verbatim}
\usepackage{mathrsfs}
\usepackage{a4wide}
\usepackage[usenames,dvipsnames]{color}
\usepackage{dsfont}
\usepackage[utf8]{inputenc}
\usepackage[T1]{fontenc}
\usepackage{enumerate}
\usepackage[mathscr]{euscript}

\DeclareFontFamily{U}{mathx}{\hyphenchar\font45}
\DeclareFontShape{U}{mathx}{m}{n}{
      <5> <6> <7> <8> <9> <10>
      <10.95> <12> <14.4> <17.28> <20.74> <24.88>
      mathx10
      }{}
\DeclareSymbolFont{mathx}{U}{mathx}{m}{n}
\DeclareFontSubstitution{U}{mathx}{m}{n}
\DeclareMathAccent{\widecheck}{0}{mathx}{"71}
\DeclareMathAccent{\wideparen}{0}{mathx}{"75}

\newtheorem{theorem}{Theorem}[section]
\newtheorem{lemma}[theorem]{Lemma}
\newtheorem{corollary}[theorem]{Corollary}
\newtheorem{proposition}[theorem]{Proposition}
\theoremstyle{remark}
\newtheorem{remark}[theorem]{Remark}
\newtheorem*{remark*}{Remark}
\theoremstyle{definition}
\newtheorem{definition}[theorem]{Definition}

\numberwithin{equation}{section}
\makeatother

\newcommand{\A}{\EuScript{A}}
\newcommand{\V}{\EuScript{V}}

\newcommand{\supp}{\operatorname{supp}}

\newcommand{\interior}{\operatorname{int}}

\newcommand{\Spec}{\operatorname{Spec}}

\newcommand{\vertiii}[1]{{\left\vert\kern-0.25ex\left\vert\kern-0.25ex\left\vert #1 
    \right\vert\kern-0.25ex\right\vert\kern-0.25ex\right\vert}}

\newcounter{smallromans}

\newenvironment{romanenumerate}
{\begin{list}{{\normalfont\textrm{(\roman{smallromans})}}}%
  {\usecounter{smallromans}\setlength{\itemindent}{0cm}%
   \setlength{\leftmargin}{5.5ex}\setlength{\labelwidth}{5.5ex}%
   \setlength{\topsep}{.5ex}\setlength{\partopsep}{.5ex}%
   \setlength{\itemsep}{0.1ex}}}%
{\end{list}}

\newcounter{smallromansdash}

{\end{list}}

\newcounter{bigromans} 
  {\end{list}}

\begin{document}
\date{November 23, 2017}
\title[The compatibility ordering on $C(X)$]{Recovering a compact Hausdorff space $X$ from the compatibility ordering on $C(X)$}
\author[T.~Kania]{Tomasz Kania}
\address{Mathematics Institute,
University of Warwick,
Gibbet Hill Rd, 
Coventry, CV4 7AL, 
England and Institute of Mathematics, Czech Academy of Sciences, \v{Z}itn\'{a} 25, 115~67 Prague 1, Czech Republic}
\email{tomasz.marcin.kania@gmail.com}
\author[M.~Rmoutil]{Martin Rmoutil}
\address{Department of Mathematics Education, Faculty of Mathematics and Physics, Charles University, Sokolovsk\'a 83, 186~75 Prague~8, Czech Republic}
\email{martin@rmoutil.eu}


\thanks{The authors acknowledge with thanks funding received from the European Research Council / ERC Grant Agreement No.~291497.}
\begin{abstract}Let $f$ and $g$ be scalar-valued, continuous functions on some topological space. We say that $g$ dominates $f$ in the compatibility ordering if $g$ coincides with $f$ on the support of $f$. We prove that two compact Hausdorff spaces are homeomorphic if and only if there exists a compatibility isomorphism between their families of scalar-valued, continuous functions. We derive the classical theorems of Gelfand--Kolmogorov, Milgram and Kaplansky as easy corollaries to our result as well as a theorem of Jarosz (\emph{Bull.~Canad.~Math.~Soc.} 1990) thereby building~a common roof for these theorems. Sharp automatic-continuity results for compatibility isomorphisms are also established.\\

\noindent \emph{Added on 30.03.2021}: Unfortunately Theorem 1.1 of the present manuscript is flawed. Erratum and addendum written jointly with D.~H.~Leung is attached. Besides providing an amendment to the said statement, it also contains more complete proof of Proposition 4.1. Theorems 1.2--1.3 are unaffected.\end{abstract}
\maketitle

\section{Introduction and the main result}

Let $X$ be a topological space and denote by $C(X)$ the family of all scalar-valued, continuous functions on $X$. In the case where $X$ is compact, $C(X)$ carries a wealth of extra structure; indeed $C(X)$ may be already viewed as a ring or an algebra (here compactness is not essential), a~metric or a Banach space when furnished with the supremum metric, a Banach or a C*-algebra when the algebraic and metric structures are combined or as a~lattice when equipped with the pointwise ordering. Over the past century, a plethora of results recovering the underlying compact space $X$ from $C(X)$ was obtained, \emph{i.e.}, results where one seeks the existence of a homeomorphism between the two compact spaces in the presence of an isomorphism between families of continuous functions in one of the above-listed categories. \smallskip

Probably the best-known result in this direction is the Gelfand--Kolmogorov theorem (\cite{gk}), which asserts that two compact Hausdorff spaces $X$ and $Y$ are homeomorphic if and only if $C(X)$ and $C(Y)$ are isomorphic as rings. The Banach--Stone theorem says that the same is true if $C(X)$ and $C(Y)$ are isometric (\cite[Th\'{e}or\`{e}me IX.4.3]{banach} in the metric case and \cite{stone} in the general case); in the light of the Mazur--Ulam theorem (\cite{mazurulam}), it is not relevant whether the isometry is linear or not. Milgram's theorem (\cite[Theorem A]{milgram}) recovers $C(X)$ from the multiplicative semigroup of $C(X)$, which means that $X$ and $Y$ are homeomorphic if and only if there exists a multiplicative bijection between $C(X)$ and $C(Y)$. Finally, a theorem of Kaplansky (\cite{kaplansky}) yields the same conclusion in the presence of a bijection between $C(X)$ and $C(Y)$ which respects the pointwise ordering. As remarked by Kaplansky himself, his result subsumes the Gelfand--Kolmogorov theorem, Milgram's theorem as well as the Banach--Stone theorem (\cite[p.~622]{kaplansky}).\smallskip

The main result of this paper is exactly in the same spirit and implies Kaplansky's theorem as an easy consequence (so that we derive also all the previously mentioned theorems). Curiously, \emph{en route} to the proof, we have managed to circumvent the need of alluding explicitly to the fact saying that maximal ideals of $C(X)$, regarded as a ring, where $X$ is a compact Hausdorff space, are kernels of evaluations at points from $X$ (this observation is crucial for most of the proofs of the previously mentioned theorems that we are aware of). In order to explain the declared main result, we require to introduce some terminology. \smallskip

Given two scalar-valued functions $f$ and $g$ on a topological space $X$, we write $$f\preceq g\text{ whenever }f(x)=g(x)\text{ for every }x\in \supp f,$$ where $\supp f$ stands for the closure of the set of points in $X$ where $f$ is non-zero (here we used a fixed scalar field either of real or complex numbers). It is readily seen that $\preceq$ is an ordering in $C(X)$ with the least element being the zero function. For brevity, we have decided to term it \textit{the compatibility ordering} and we do hope this name is self-explanatory. This ordering is therefore another piece of structure that may be imposed on $C(X)$. In addition to the ordering itself, we shall be concerned with morphisms between families of continuous functions which preserve it (without assuming any extra properties such as linearity or continuity). Thus, a map $T\colon C(X)\to C(Y)$ is a \emph{compatibility morphism} if $Tf\preceq Tg$ whenever $f\preceq g$ for $f,g\in C(X)$. We will call bijective compatibility morphisms whose inverses are also compatibility morphisms \emph{compatibility isomorphisms}. In other words, compatibility isomorphisms are isomorphisms between families of continuous functions regarded as partially ordered sets when furnished with the compatibility ordering. \smallskip

We are now in a position to phrase the first main result of the paper. The proofs are postponed to Section~\ref{mainproof}.
\begin{theorem}\label{mainish}Let $X$ and $Y$ be compact Hausdorff spaces. Then $X$ and $Y$ are homeomorphic if and only if there exists a compatibility isomorphism between $C(X)$ and $C(Y)$. \end{theorem}

Theorem~\ref{mainish} nicely complements a recent result of Li and Wong (\cite{lw2014}), where the authors recover a compact Hausdorff space $X$ from linear bijections that preserve functions which do not assume the value zero. It follows from the considerations in the previous paragraph that compatibility isomorphisms ignore such functions and recover the space from functions which \emph{do} assume the value zero.\smallskip

At this point it is reasonable to ask whether compatibility isomorphisms may be described by formulae involving the homeomorphism produced in the proof of Theorem~\ref{mainish} and some other auxiliary functions. Sometimes this can be done, as we prove in Section~\ref{consequences} that, for example, ring isomorphisms $T\colon C(X)\to C(Y)$ are compatibility morphisms and it is known that every such isomorphism is given by composition with a homeomorphism between $Y$ and $X$. We also prove that multiplicative bijections and pointwise-ordering isomorphisms are compatibility isomorphisms too. Milgram has observed in \cite{milgram} that if $X$ and $Y$ are compact Hausdorff spaces, $\varphi\colon X\to Y$ is a homeomorphism and $p$ is a~positive function on $X$, then the mapping $T\colon C(Y)\to C(X)$ defined by
$$(Tf)(z) = |f(\varphi(z))|^{p(z)}\cdot {\rm sgn} f(\varphi(z))\quad (f\in C(Y), z\in X)$$
is multiplicative and if $X$ does not have isolated points, then every multiplicative bijection is of this form. (See \cite{lesnjak} for the description of general continuous multiplicative maps between algebras of continuous functions.) Kaplansky (\cite[p.~629]{kaplansky2}) exhibited a discontinuous pointwise-ordering automorphism of $C(\beta \mathbb{N})$, where $\beta \mathbb{N}$ is the \v{C}ech--Stone compactification of the discrete space of natural numbers. However Kaplansky's paper also contains certain positive results (for example, he established an automatic-continuity result in the case where $X$ is first-countable). Thus, our results constitute a common roof for many classical results in the area, yet in general it is impossible to express an \emph{a priori} given compatibility isomorphism by a more concise formula. This ought to be therefore regarded as the cost of the aimed generality. \smallskip 

Our notion of a compatibility morphism is much more general as for every connected compact Hausdorff space $X$ (including the one-point space) there exist at least $2^{\mathfrak{c}}$ distinct compatibility automorphisms of $C(X)$. To see this, choose a self-bijection $\Phi$ of the set $${\rm GL}(C(X))=\{f\in C(X)\colon f(x)\neq 0\text{ for all }x\in X\}.$$ (This set has cardinality at least continuum so it has at least $2^\mathfrak{c}$ self-bijections.) We define $T$ by setting $Tf = f$ if $f\in C(X)\setminus {\rm GL}(C(X))$ and $Tf = \Phi(f)$ otherwise. Then there are $2^{\mathfrak{c}}$ such discontinuous maps $T$.\smallskip

Quite surprisingly from the point of view of automatic-continuity theory, compatibility isomorphisms from $C(X)$ for certain infinite compact spaces are automatically continuous even though this is not the case when one considers the simplest possible compact Hausdorff space, \emph{i.e.}, the one-point space. To state the results we employ the notion of a~component of a topological space. \smallskip

Let $X$ be a topological space and $x\in X$. The \emph{component of }$x$ is the union of all connected subspaces of $X$ which contain $x$; the component of $x$ is connected itself (we consider singletons to be connected subsets of $X$). Components of two distinct points in $X$ are either equal or disjoint so components of points $X$ form a decomposition of $X$ into pairwise disjoint sets. For brevity of notation we call a \emph{component of }$X$ the component of some point in $X$.

\begin{theorem} \label{T:ContPos}
Let $X$ and $Y$ be compact Hausdorff spaces and suppose that $X$ is sequentially compact and that every component of $X$ is nowhere dense. Then every compatibility isomorphism $T\colon C(X)\to C(Y)$ is norm-continuous.
\end{theorem}

We point out that the prototypical example of a compact space $X$ which meets the hypotheses of the above theorem is the Cantor set $\Delta$, however there are numerous further examples that are not zero-dimensional too, just to mention spaces of the form $\Delta\times Z$, where $Z$ is any connected, compact metric space. \smallskip


Theorem~\ref{T:ContPos} is in a sense optimal as there always exists a discontinuous compatibility isomorphism whenever the space a has a component with non-empty interior, as stated in the following theorem.

\begin{theorem} \label{T:Discont}
Let $X$ be a compact Hausdorff space which has at least one component with a non-empty interior. Then there exists a compatibility isomorphism $T\colon C(X)\to C(X)$ which is not continuous as a map on $C(X)$ either with the norm topology or the topology of pointwise convergence.
\end{theorem}

Uncountable chains with respect to the compatibility ordering have been investigated by the first-named author together with Hart and Kochanek (\cite{hkk}) and together with Smith (\cite{kasm}) in relation to weakly compact operators on spaces of continuous functions. Krupski then worked on and solved in his PhD thesis (\cite{krupski}) certain problems left open in \cite{hkk}. The authors are not aware of any other occurrences (neither explicit nor implicit) of the compatibility ordering in the literature. Having thus entered a virgin territory, we find ourselves obliged to uncover some basic properties of the compatibility ordering.\bigskip

\noindent \textbf{Acknowledgments}. We are indebted to Juan Francisco Camasca Fern\'{a}ndez (S\~ao Paulo) and Luiz Gustavo Cordeiro (Ottawa) for attentive reading of the preliminary version of the manuscript and for detecting certain slips that were lurking in the draft. We also wish to express our deepest gratitude to Luiz Gustavo Cordeiro who managed to discover a~substantial gap in the proof of the main theorem. He also provided us with detailed notes on a possible fix, convincing us to return to our original idea of employing lattice theory. In addition to that, he gave us his kind permission to use and publish some of his elegant ideas present in the current version of the first proof of Theorem~\ref{mainish}.

\section{Consequences of Theorem~\ref{mainish}}\label{consequences}

Before we list the promised consequences of Theorem~\ref{mainish}, let us make some preliminary observations. First of all, we note that compatibility isomorphisms map the zero function to the zero function. Secondly, we observe that if $f,g\in C(X)$ for some topological space $X$, then $f\preceq g$ implies that $fg=f^2$. The converse also holds so we have got the following proposition. 

\begin{proposition}\label{mult}Let $X$ be a topological space and suppose that $f,g\in C(X)$. Then $f\preceq g$ if and only if $fg=f^2$.\end{proposition}
\begin{proof}We only need to show $(\Leftarrow)$. Suppose that $fg=f^2$. Pick $x\in \supp f$. Then there exists a net $(x_\alpha)_{\alpha\in \Lambda}$ in $X$ which converges to $x$ and such that $f(x_\alpha)\neq 0$ for all $\alpha\in \Lambda$. As $fg=f^2$, for each $\alpha\in \Lambda$ we have $f(x_\alpha)g(x_\alpha)= f(x_\alpha)f(x_\alpha)$ so $f(x_\alpha) = g(x_\alpha)$. By continuity of $f$ and $g$, we conclude that $f(x)=g(x)$.\end{proof}
Armed with this simple but useful algebraic description of the compatibility ordering, we are now ready to list some consequences of Theorem~\ref{mainish}.\bigskip

The theorems of Gelfand--Kolmogorov (\cite{gk}) and Milgram (\cite{milgram}) are immediate consequences of  Theorem~\ref{mainish} combined with Proposition~\ref{mult} because every bijection between $C(X)$ and $C(Y)$ which preserves multiplication (for instance, a ring isomorphism) necessarily preserves the compatibility ordering and the inverse of a bijection which preserves multiplication preserves multiplication too. Let us record these consequences formally.

\begin{corollary}[The Gelfand--Kolmogorov theorem]Let $X$ and $Y$ be compact Hausdorff spaces. Then $X$ and $Y$ are homeomorphic if and only if $C(X)$ and $C(Y)$ are isomorphic as rings. \end{corollary}

\begin{corollary}[Milgram's theorem]Let $X$ and $Y$ be compact Hausdorff spaces. Then $X$ and $Y$ are homeomorphic if and only if there exists a multiplicative bijection between $C(X)$ and $C(Y)$. \end{corollary}

We specialise now to the case of real scalars. We define for a function $f\colon X\to \mathbb{R}$, the positive and the negative part of $f$, by setting $f^+= \max\{f, 0\}$ and $f^-=-\min\{f, 0\}$, respectively, pointwise. Of course, if $X$ is a topological space and $f$ is continuous, then so are $f^+, f^-$. Moreover, the pointwise ordering on $C(X)$ makes it a lattice with lattice operations $\min\{f,g\}$ and $\max\{f,g\}$ ($f,g\in C(X)$) defined pointwise. \smallskip

Let us then record the following simple fact which links the usual pointwise ordering with the compatibility ordering.
\begin{lemma}\label{plusminus} Let $X$ be a topological space and suppose that $f,g\in C(X)$. Then 
\begin{romanenumerate}
\item\label{easy1} $f\preceq g$ if and only if $f^+\preceq g^+$ and $f^-\preceq g^-$,
\item\label{easy2} if $f,g\geqslant 0$, then $f\preceq g$ if and only if $f\leqslant g$ and $\max\{g-f, f\}\geqslant g$,
\item\label{easy3} if $f,g\leqslant 0$, then $f\preceq g$ if and only if $f\geqslant g$ and $\min\{g-f, f\}\leqslant g$.
\end{romanenumerate}\end{lemma}
\begin{proof}\eqref{easy1} is immediate. For \eqref{easy2}, suppose that $f,g$ are non-negative. If $f\preceq g$, then as $f$ and $g$ are non-negative we must have $f\leqslant g$. If $x\in\supp f$, we have $f(x)=g(x)$ so $\max\{g(x)-f(x), g(x)\}=g(x)$. Otherwise, $\max\{g(x)-0, 0\}=g(x)$, hence we are done.\smallskip

Conversely, suppose that $f\leqslant g$ and $\max\{g-f, f\}\geqslant g$. Let us pick $x\in\supp f$. We have $f(x) \leqslant g(x)$. If $f(x)>0$, then $g(x)-f(x) < g(x)$. As $\max\{g(x)-f(x), f(x)\}\geqslant g(x)$ we must have $f(x)\geqslant g(x)$, so $f(x)=g(x)$. By continuity, $f(x)=g(x)$ for $x\in \supp f$.\smallskip

The proof of \eqref{easy3} is completely analogous.\end{proof}

We are now ready to show that Theorem~\ref{mainish} implies Kaplansky's theorem which recovers a compact Hausdorff space from the pointwise ordering on the family of all real-valued continuous functions on it. To this end, it is enough to show that lattice isomorphisms between $C(X)$ and $C(Y)$ (where $X$ and $Y$ are compact Hausdorff spaces) are translations of compatibility morphisms by the value at 0. Since $C(X)$ is a lattice under the pointwise ordering, a bijective order homomorphism $T\colon C(X)\to C(Y)$ is a lattice isomorphism, \emph{i.e.}, it preserves the lattice operations:
$$T(\max\{f,g\}) = \max\{Tf, Tg\}\text{ and }T(\min\{f,g\}) = \min\{Tf, Tg\}\quad (f,g\in C(X))$$
and the inverse of $T$ also has this property (see Lemma~\ref{latticesaresuperb}).

\begin{proposition}Let $X$ and $Y$ be compact Hausdorff spaces and let $S\colon C(X)\to C(Y)$ be a lattice isomorphism. Then $Tf=Sf - S0$ $(f\in C(X))$ is a lattice isomorphism between $C(X)$ and $C(Y)$ which is also a compatibility isomorphism.\end{proposition}

\begin{proof}Certainly, $T$ is a lattice isomorphism as
$$ Sf - S0 \leqslant Sg - S0 \iff Sf \leqslant Sg \iff f\leqslant g\quad \big(f,g\in C(X)\big).$$

Since $T^{-1}$ is also a (bijective) lattice homomorphism, it is enough to show that for $f,g\in C(X)$ with $f \preceq g$ one has $Tf\preceq Tg$. By Lemma~\ref{plusminus}\eqref{easy1}, it is enough to consider the case where $f,g$ are both non-negative (then so are $Tf$ and $Tg$). \smallskip

So suppose that $f,g\geqslant 0$ and $f\preceq g$. Then we clearly have that
\begin{equation*}
0\leqslant g-f\leqslant g,\; f\leqslant g,\;
\max\{g-f, f\}\geqslant g \text{ and } \min\{g-f, f\}\leqslant 0. 
\end{equation*}
Consequently, since $T$ is a pointwise-ordering isomorphism with $T(0)=0$, it respects these inequalities, which means that 
\begin{gather*}\label{simpleinequalities}
0\leqslant T(g-f)\leqslant Tg,\; 0\leqslant Tf\leqslant Tg;\\ \max\{T(g-f), Tf\}\geqslant Tg \text{ and } \min\{T(g-f), Tf\} \leqslant 0.
\end{gather*}
\smallskip
We obtain that $\max\{T(g-f),Tf\}=Tg$ and it also follows that the functions $T(g-f)$ and $Tf$ are orthogonal (\emph{i.e.}, $T(g-f)\cdot Tf =0$) as their minimum is equal to the zero function. Hence, $Tg=\max\{T(g-f),Tf\}=T(g-f)+Tf$, and one can readily see that this means $Tf\preceq Tg$. Indeed, the orthogonality of $T(g-f)$ and $Tf$ together with the continuity of $T(g-f)$ imply that $T(g-f)$ vanishes in the support of $Tf$.
\end{proof}

Consequently, we obtain Kaplansky's theorem.

\begin{corollary}[Kaplansky's theorem]\label{kaplan}Let $X$ and $Y$ be compact Hausdorff spaces. Then $X$ and $Y$ are homeomorphic if and only if there exists bijection between $C(X)$ and $C(Y)$ which preserves the pointwise ordering. \end{corollary}

A (possibly non-linear) map $T\colon C(X)\to C(Y)$ is called \emph{disjointness preserving} provided that $Tf \cdot Tg = 0$ whenever $f \cdot g=0$.  Lemma~\ref{additive} says that compatibility isomorphisms are disjointness preserving. Jarosz studied linear, disjointness preserving maps between spaces of continuous functions on compact Hausdorff spaces and proved that every such map is a weighted composition operator (\cite{jarosz}, see also \cite[Theorem 3.1]{aj}). He also obtained the following result, which now becomes another corollary to Theorem~\ref{mainish}.

\begin{corollary}[Jarosz's theorem]Let $X$ and $Y$ be compact Hausdorff spaces. Then $X$ and $Y$ are homeomorphic if and only if there exists a linear, disjointness preserving bijection between $C(X)$ and $C(Y)$.\end{corollary}
\begin{proof}Let $T\colon C(X)\to C(Y)$ be a disjointness preserving linear bijection. Inverses of such maps are automatically linear and disjointness preserving (\cite{hp}, \cite[Theorem 3.6]{koldunov}), so in the light of Theorem~\ref{mainish}, it is enough to show that $T$ is a compatibility morphism. To this end, fix two functions $f,g\in C(X)$ such that $f\preceq g$. Then $(g-f)f=0$, so $T(g-f)$ and $Tf$ are orthogonal, \emph{i.e.}, $T(g-f)\cdot Tf=0$. Thus, by linearity, $Tg =T(g-f) +Tf \succeq Tf$. \end{proof}

\section{Auxiliary results and the proof of Theorem~\ref{mainish}}\label{mainproof}
We call two scalar-valued functions $f,g$ defined on the same space \emph{orthogonal} when $f\cdot g=0$. This, of course, does not necessarily mean that the supports of $f$ and $g$ are disjoint.

\begin{lemma}\label{additive}
Let $X$ and $Y$ be topological spaces. Suppose that $T\colon C(X)\to C(Y)$ is a~compatibility isomorphism and $f,g\in C(X)$. Consider the following conditions:
\begin{romanenumerate}
\item\label{d1} $f$ and $g$ are orthogonal,
\item\label{d2} $Tf$ and $Tg$ are orthogonal,
\item\label{d3} $T(f+g)=Tf + Tg  $.
\end{romanenumerate}
Then \eqref{d1} and \eqref{d2} are equivalent and imply \eqref{d3}.
\end{lemma}

\begin{proof}\eqref{d1} $\implies$ \eqref{d2}. Suppose that $f, g\in C(X)$ are orthogonal functions. We then have $f,g\preceq f+g$, whence $Tf, Tg \preceq T(f+g)$, and setting $A:=\supp Tf\cap \supp Tg$ we obtain that $Tf(x)=Tg(x)$ for every $x\in A$. Define a function $\varphi$ on $Y$ by $\varphi := (Tf)\cdot \mathds{1}_A,$ where $\mathds{1}_A$ stands for the indicator function of $A$. We \emph{claim} that $\varphi$ is continuous on $Y$.\smallskip

To see this, let us first note that $\varphi = (Tf)\cdot \mathds{1}_{\supp Tg}$. In order to demonstrate the continuity of $\varphi$ it is therefore enough to show that $Tf$ vanishes on $\partial \supp Tg$, the boundary of $\supp Tg$. To this end, fix $x\in \partial\supp Tg$; then $Tg(x)=0$. Let $(x_\alpha)_{\alpha\in \Lambda}$ be a net in $Y$ that converges to $x$ such that $Tg(x_\alpha)\neq 0$ for each $\alpha\in \Lambda$. As $Tg\preceq T(f+g)$, we see that $Tg(x_\alpha) = (T(f+g))(x_\alpha)$ for each $\alpha\in \Lambda$. Consequently,
$$0 = Tg(x) = \lim_{\alpha} Tg(x_\alpha) = \lim_{\alpha} (T(f+g))(x_\alpha) = (T(f+g))(x).$$
The conclusion now follows as we also have $Tf\preceq T(f+g)$. \smallskip

Now, we observe that $\varphi\preceq Tf, Tg$ and so $T^{-1}(\varphi) \preceq f,g$. Recalling that $f$ and $g$ are orthogonal, we infer that $T^{-1}(\varphi)=0$ as $0$ is the only element of $C(X)$ that is dominated in the compatibility ordering by two orthogonal functions. Thus, $\varphi=0$ as each compatibility isomorphism maps $0$ to $0$. Consequently, $A$ has empty interior (otherwise we would have found a point $x\in A$ such that $\varphi(x)\neq 0$) and thus $(Tf)\cdot (Tg)=0$.\smallskip

\eqref{d2} $\implies$ \eqref{d1}. Apply the previous implication to $T^{-1}$.

\eqref{d2} $\implies$ \eqref{d3}. By the above, we know that whenever $Tf$ and $Tg$ are orthogonal, then so are $f$ and $g$. From these facts we clearly have
$$\sup\nolimits_{\preceq}\{f, g \}=f+g \quad\text{and}\quad \sup\nolimits_{\preceq}\{Tf, Tg\}= Tf+Tg,$$
and the fact that $T$ is a compatibility isomorphism easily implies that
$$T\left(\sup\nolimits_{\preceq}\{f, g \}\right)= \sup\nolimits_{\preceq}\{Tf, Tg\}.$$
The desired conclusion follows.\end{proof}

In general \eqref{d3} does not imply \eqref{d1} because the identity map on $C(X)$ is plainly an additive compatibility isomorphism. We remark in passing that should the infimum $\inf\nolimits_{\preceq}\{f, g\}$ exist ($f,g\in C(X)$), it is preserved by $T$.
\begin{remark}\label{rho}Let $Y$ be a compact metric space which contains at least two points and let $X$ be the one-point space. For $f,g\in C(X)$ one has $f\preceq g$ if and only if either $f=g$ or $f=0$. Let us then fix a bijection $T\colon C(X)\to C(Y)$ which maps the zero function to the zero function. Then $T$ is a compatibility morphism whose inverse is not a compatibility morphism. Thus, the requirement that the inverse of a~compatibility isomorphism be a compatibility morphism is not redundant.\end{remark}
\begin{definition}Let $X$ be a topological space. For $f\in C(X)$ we define $$\sigma(f) = \interior\supp f,$$
the interior of the support of $f$, and
$$\varrho(f) = \overline{\interior f^{-1}(0)}.$$
\end{definition}
\begin{remark}\label{stupidremark}We note that $\varrho(f) = X \setminus \sigma(f)$ for $f\in C(X)$. Indeed,
$$ \sigma(f) =  \interior \overline{X\setminus f^{-1}(0)} = \interior (X \setminus \interior f^{-1}(0)) = X \setminus \overline{\interior f^{-1}(0)} = X\setminus \varrho(f). $$

\end{remark}
\begin{remark}\label{clause4}
We can now add another clause to Lemma~\ref{additive} which is equivalent to \eqref{d1}:
\begin{romanenumerate}
\item[(iv)]\label{d4} $\sigma(f)\cap\sigma(g) = \varnothing$.
\end{romanenumerate}
Obviously (iv) $\implies$ \eqref{d1}; to see that the opposite implication holds as well, assume that $\sigma(f)\cap\sigma(g)\neq\varnothing$. Let $U=\{x\in X\colon f(x)\neq 0\}$ and $V=\{x\in X\colon g(x)\neq 0\}$. Then
$$U\cap V\supseteq \big(\sigma(g)\cap\sigma(f)\big)\setminus(\partial U \cup \partial V) \neq \varnothing$$ as boundaries of open sets are nowhere dense. Consequently, $f\cdot g$ assumes non-zero values.
\end{remark}
The following lemma is a standard fact from point-set topology. We include a proof for the reader's convenience.

\begin{lemma}\label{base}Let $X$ be a completely regular space. Then the family $\{\sigma(f)\colon f\in C(X)\}$ forms an open base for $X$.\end{lemma}
\begin{proof}Let $V$ be a non-empty open subset of $X$ and let $x\in V$. We assumed $X$ to be (completely) regular so there exists an open neighbourhood $U$ of $x$ such that $\overline{U}\subseteq V$. By complete regularity, there exists $f\in C(X)$ such that $f(x)=1$ and $f(y)=0$ for $y\in X\setminus U$. Then $\sigma(f)\subseteq \supp f\subseteq \overline{U} \subseteq V$ and we are done. \end{proof}

\begin{proposition}\label{wd}Let $X$ and $Y$ be completely regular spaces and let $T\colon C(X)\to C(Y)$ be a~ compatibility isomorphism. If $f,g\in C(X)$ and $\sigma(f) \subseteq \sigma(g)$, then $\sigma(Tf) \subseteq \sigma(Tg)$. In particular, $\sigma(f) = \sigma(g)$ if and only if $\sigma(Tf) = \sigma(Tg)$.\end{proposition}

\begin{proof}Without loss of generality we may take a non-zero function $f\in C(X)$ as $\sigma(T0)$ is empty. Assume, in search of a contradiction, that $\sigma(f) \subseteq \sigma(g)$ and $\sigma(Tf) \not\subseteq \sigma(Tg)$. We note that if $ \sigma(Tf) \subseteq {\rm supp}\, Tg$, then, by openness of $\sigma(Tf)$,  $\sigma(Tf)\subseteq \sigma(Tg)$, however we have excluded this possibility.\smallskip

Set $V=\sigma(Tf) \setminus \supp Tg$. Then $V$ is a non-empty open set as $\sigma(Tf)\subsetneq \sigma(Tg)$, hence $\sigma(Tf)\subsetneq {\rm supp}(Tg)$. Pick any point $v\in V$ with $Tf(v)\neq 0$ (the set of such points is dense in $V$) and choose a function $h\in C(X)$ such that $Th(v)=1$ and $\supp Th\subseteq V$ (such a~function exists because $T$ is a bijection and $Y$ is completely regular). Then $Th$ and $Tg$ are orthogonal; by Lemma~\ref{additive}, $h$ and $g$ must be orthogonal too. Since $\sigma(f)\subseteq\sigma(g)$, we also have that $h$ and $f$ are orthogonal and one more application of Lemma~\ref{additive} gives us that the same must be true for $Th$ and $Tf$, which contradicts $Th(v)\cdot Tf(v)\neq 0$.
\end{proof}

\begin{proposition}\label{P:tau} Let $X$ and $Y$ be completely regular spaces and let $T\colon C(X)\to C(Y)$ be a~compatibility isomorphism. Then the map
$$\tau(\sigma(f)) = \sigma(Tf) \quad \big(f\in C(X)\big)$$
is a (well-defined) bijection between $\{\sigma(f)\colon f\in C(X)\}$ and $\{\sigma(g)\colon g\in C(Y)\}$, which preserves the inclusion.
\end{proposition}

\begin{proof}
Suppose that $f,g\in C(X)$ are such that $\sigma(f)=\sigma(g)$. Then, by Proposition~\ref{wd}, $\sigma(Tf)=\sigma(Tg)$ so $\tau$ is a well-defined map. \smallskip

To demonstrate that the map $\tau$ is surjective, take $g\in C(Y)$ and so $T^{-1}g\in C(X)$. Then $\sigma(g)=\tau(\sigma(T^{-1}g))$ is in the range of $\tau$. Define $\iota\colon\{\sigma(g)\colon g\in C(Y)\}\to\{\sigma(f)\colon f\in C(X)\}$ by
$$\iota\big(\sigma(g)\big)=\sigma(T^{-1}g).$$
Then $\iota$ is well-defined by the same argument that we have used for $\tau$, and it is clear that $\iota=\tau^{-1}$. Therefore $\tau$ is injective. Proposition~\ref{wd} also ensures us that both $\tau$ and $\iota=\tau^{-1}$ preserve inclusion, and the proof is complete.
\end{proof}

In the light of Remark~\ref{stupidremark}, the following is an immediate corollary to Proposition~\ref{P:tau}.

\begin{corollary}\label{P:vartheta} Let $X$ and $Y$ be completely regular spaces and let $T\colon C(X)\to C(Y)$ be a~compatibility isomorphism. Then the map
$$\vartheta(\varrho(f)) = \varrho(Tf) \quad \big(f\in C(X)\big)$$
is a well-defined bijection between $\{\varrho(f)\colon f\in C(X)\}$ and $\{\varrho(g)\colon g\in C(Y)\}$, which preserves the inclusion.
\end{corollary}

\subsection{Regularly open and regularly closed sets}

Let $X$ be a completely regular space. A subset $A\subseteq X$ is \emph{regularly open} if $A = \interior \overline{A}$ and \emph{regularly closed} if $A=\overline{\interior A}$. It is a~triviality that a set is regularly closed if and only if its complement is regularly open. For example, for every $f\in C(X)$, the set $\sigma(f)$ is regularly open (moreover, $\overline{\sigma(f)}={\rm supp}\, f$). As it is well-known and easy to verify, the family ${\rm RO}(X)$ consisting of all regularly open sets forms a complete, distributive lattice with operations
$$U\vee_{{\rm ro}} V = \interior \overline{U\cup V}\text{ and }U \wedge_{{\rm ro}} V = U\cap V.$$
Similarly, the family ${\rm RC}(X)$ of all regularly closed sets forms a complete, distributive lattice with operations
$$F\vee_{{\rm rc}} G = F\cup G \text{ and } F \wedge_{{\rm rc}} G= \overline{\interior F\cap G}.$$

\begin{proposition}Let $X$ be a completely regular space. Then for any $f,g\in C(X)$ we have
$$\sigma(f)\vee_{{\rm ro}} \sigma(g) = \sigma(|f|+|g|)\text{ and } \sigma(f)\wedge_{{\rm ro}} \sigma(g) = \sigma(fg)$$
and
$$\varrho(f)\wedge_{{\rm rc}} \varrho(g) = \varrho(|f|+|g|)\text{ and } \varrho(f)\vee_{{\rm rc}} \varrho(g) = \varrho(fg).$$
\end{proposition}
\begin{proof}Given the duality between the operations $\sigma$ and $\varrho$, that is, the relation
$$\begin{array}{lclcl}X\setminus \big(\sigma(f) \vee_{{\rm ro}} \sigma(g)\big)& =& \overline{X\setminus \overline{\sigma(f)\cup \sigma(g)}} & = & \overline{X\setminus \overline{\varrho(f)\cap \varrho(g)}} \\
& = & \overline{{\rm int}\, \varrho(f)\cap \varrho(g)} & = & \varrho(f)\wedge_{{\rm rc}} \varrho(g)\end{array}\qquad (f,g\in C(X)), $$
it is enough to prove the latter two equalities only. \smallskip

We have
$$\varrho(f)\wedge_{{\rm rc}} \varrho(g)\subseteq  \varrho(f)\cap \varrho(g) = \overline{\interior f^{-1}(0)} \cap \overline{\interior g^{-1}(0)}\subseteq f^{-1}(0)\cap g^{-1}(0) = (|f|+|g|)^{-1}(0).$$
For the converse inclusion, it is enough to notice that $$(|f|+|g|)^{-1}(0)\subseteq f^{-1}(0)\text{ and }(|f|+|g|)^{-1}(0)\subseteq g^{-1}(0),$$ 
so by performing first the operation of taking the interior  and then taking the closure, the conclusion follows (this is clear, \emph{e.g.}, from the fact that the lattice ordering on ${\rm RC}(X)$ is the inclusion---\emph{cf.} Remark~\ref{miracle}).\smallskip

As for the second equality, the inclusion $\varrho(f)\vee_{{\rm rc}} \varrho(g)\subseteq \varrho(fg)$ is clear for $\varrho(f)\subseteq \varrho(fg)$ and $\varrho(g)\subseteq \varrho(fg)$. In order to prove the converse inequality, set $$V=\interior (fg)^{-1}(0)\setminus \varrho(f).$$ We will be finished when we prove that $V\subseteq\interior g^{-1}(0)$. Indeed, in this case we shall have $\interior(fg)^{-1}(0)\subseteq  \varrho(f)\cup\varrho(g)$. To see that $V\subseteq\interior g^{-1}(0)$, take any $x\in V$. Clearly, for any open neighbourhood $U\subseteq V$ of $x$ we have points $y\in U\setminus f^{-1}(0)$ because otherwise we would have $U\subseteq \interior f^{-1}(0)$, which contradicts $V\cap\varrho(f)=\varnothing$. However $y\in V\subseteq\interior (fg)^{-1}(0)$, so $g(y)=0$ (as $f(y)\neq 0$), and it follows from the continuity of $g$ that $g(x)=0$. Thus $V\subseteq g^{-1}(0)$, and the openness of $V$ now implies the desired conclusion.
\end{proof}

We note that $\sigma(0\cdot \mathds{1}_X) = \varnothing = \varrho(\mathds{1}_X)$ and $\sigma(\mathds{1}_X) = X = \varrho(0 \cdot \mathds{1}_X)$, hence both lattices $\{\sigma(f)\colon f\in C(X)\}$ and $\{\varrho(f)\colon f\in C(X)\}$ have neutral elements with respect to their operations of join and meet, that is to say, they are \emph{bounded}. 

\begin{corollary}Let $X$ be a completely regular space. Then the families $\{\sigma(f)\colon f\in C(X)\}$ and $\{\varrho(f)\colon f\in C(X)\}$ are sublattices of the lattice of all regularly open and regularly closed sets, respectively. Consequently, they are {a posteriori}, bounded distributive lattices. \end{corollary}

\subsection{A detour to abstract lattice theory} Let $(A, \wedge, \vee)$ be an abstract lattice. Then $A$ is also a partially ordered set when equipped with the lattice ordering \begin{equation}\label{connectinglemma}a\leqslant b \iff a \wedge b = a \iff a \vee b = b \quad (a,b\in A)\end{equation}
(see, \emph{e.g.}, \cite[Lemma~2.8]{dp}). It is a remarkable property of lattices that in a sense, the lattice order remembers the lattice structure. More formally, we have the following characterisation of lattice isomorphisms (see, \emph{e.g.}, \cite[Lemma~2.19]{dp}).
\begin{lemma}\label{latticesaresuperb}Let $A$ and $B$ be lattices and let $\vartheta\colon A\to B$ be a bijection. Then $\vartheta$ is a~lattice isomorphism if and only if $\vartheta$ respects the lattice ordering, that is to say, it satisfies the relation $$a \leqslant b \Longrightarrow \vartheta(a)\leqslant \vartheta(b)\quad (a,b\in A).$$ \end{lemma}
\begin{remark}\label{miracle}It is to be noted that the join operation $\wedge_{\rm ro}$ in the lattice of regularly open sets and the meet operation $\vee_{\rm rc}$ in the lattice of regularly closed sets are nothing but set-theoretic intersection and union, respectively, so by \eqref{connectinglemma}, the lattice ordering in either lattice is just the ordinary inclusion of sets.\end{remark}\smallskip

Let $A$ be a lattice. A proper subset $\mathscr{U}$ of $A$ is a \emph{prime filter} if for $a\in A$ and $b\in \mathscr{U}$ such that $b\leqslant a$ we have $a\in \mathscr{U}$, $a\wedge b \in  \mathscr{U}$ for all $a,b\in \mathscr{U}$ and at least one of the elements $a,b\in A$ belongs to $\mathscr{U}$ whenever $a\vee b\in \mathscr{U}$. Let us denote by $\Spec A$ \emph{the spectrum} of $A$, that is the set of all prime filters of $A$.\smallskip

Given a bounded, distributive lattice $A$, the spectrum of $A$ may be topologised by the base
\begin{equation}\label{baseets}U_a:=\{\EuScript{U}\in \Spec A\colon a\notin \mathscr{U}\}\quad (a\in A). \end{equation}
This is a base for a topology $\EuScript{G}$ on $\Spec A$. To see this note first that $\Spec A = U_0$. Secondly, for $a,b\in A$ one has $U_a \cap U_b = U_{a \vee b}$. Indeed $U_a \cap U_b = \{ \EuScript{U}\in \Spec A\colon a\notin \mathscr{U} \text{ and } b\notin \mathscr{U}   \}$. However for a given $\mathscr{U}\in \Spec A$, by primarity, $a\vee b\in \mathscr{U}$ if and only if $a\in \mathscr{U}$ or $b\in \mathscr{U}$. Consequently, $U_a \cap U_b = U_{a \vee b}$. \medskip

\noindent \emph{Note.} The just-defined topology $\EuScript{G}$ is not the same as the usual Zariski topology on $\Spec A$, which is defined as the topology generated by the base $\{\EuScript{U}\in \Spec A\colon a\in \mathscr{U}\}$ $(a\in A)$.\smallskip

Every ultrafilter in a bounded distributive lattice is prime (\cite[Theorem 10.11]{dp}), hence the set ${\rm Ult}\,(A)$ comprising all ultrafilters in $A$ is a subset of the spectrum of $A$ (in general proper), therefore it inherits the topology $\EuScript{G}$ from $\Spec A$.\smallskip

\begin{remark}\label{compactlattice}Let $X$ be a compact Hausdorff space. Since the family $\{\sigma(f)\colon f\in C(X)\}$ is an open base for $X$ (Lemma~\ref{base}), the family $\Theta(X):=\{\varrho(f)\colon f\in C(X)\}$ is a closed base for $X$, which means that every closed subset of $X$ is the intersection of some subfamily of $\Theta(X)$. By compactness of $X$, every ultrafilter $\EuScript{U}$ in the lattice $\Theta(X)$ converges. By the Hausdorff property of $X$, every ultrafilter converges to a single point $x\in X$, which means that $\EuScript{U}=\{F\in \Theta(X)\colon x\in F\}$. \end{remark}

For a point $x\in X$ we set $\mathscr{U}_x = \{F\in \Theta(X) \colon x\in F\}$. By Remark~\ref{compactlattice}, every ultrafilter in $\Theta(X)$ is of the form $\mathscr{U}_x$ for some $x\in X$. The next result is inspired by Wallman's theorem (\cite{wallman}) and is probably well-known.

\begin{proposition}\label{pi}Let $X$ be a compact Hausdorff space. Then the map $\Upsilon\colon X\to {\rm Ult}\, \Theta(X)$ given by $$\Upsilon(x)= \mathscr{U}_x\; (x\in X),$$ is a homeomorphism.   \end{proposition}

\begin{proof}By Lemma~\ref{base}, the family $\{\sigma(f)\colon f\in C(X)\}$ is an open base for $X$. Let us recall that $\EuScript{G}$ is generated by the base $\{\EuScript{U}\in {\rm Ult}\, \Theta(X)\, \colon F \notin \mathscr{U}\}$ $(F\in \Theta(X))$. Taking into account the fact that for $x\in X$ and $f\in C(X)$ we have
$$x\in \sigma(f) \iff x\notin \varrho(f) \iff \varrho(f) \notin \Upsilon(x) \iff \Upsilon(x) \in U_{\varrho(f)},  $$
we arrive at the conclusion that $\Upsilon$ maps basic open sets onto basic open sets of $\EuScript{G}$, hence it is a homeomorphism. \end{proof}

\subsection{Proofs of Theorem~\ref{mainish}} We are now in a position to provide a short proof the first main result of the paper in the form of a string of previously established results.

\begin{proof}[Proof of Theorem~\ref{mainish}]Let $T\colon C(X)\to C(Y)$ be a compatibility isomorphism. By Corollary~\ref{P:vartheta}, the map $\vartheta(\varrho(f)) = \varrho(Tf)$ ($f\in C(X)$) is an inclusion-preserving bijection between $\Theta(X)$ and $\Theta(Y)$. Since the ordinary inclusion is the lattice ordering in $\Theta(X)$ and $\Theta(Y)$ (Remark~\ref{miracle}), by Lemma~\ref{latticesaresuperb}, $\vartheta$ is a lattice isomorphism between $\Theta(X)$ and $\Theta(Y)$. Consequently, ${\rm Ult}\, \Theta(X)$ and ${\rm Ult}\, \Theta(Y)$ are homeomorphic. By Proposition~\ref{pi}, $X$ is homeomorphic to ${\rm Ult}\, \Theta(X)$ and $Y$ is homeomorphic to ${\rm Ult}\, \Theta(Y)$ so the conclusion follows.\end{proof}

Let us remark that another endgame in the proof of Theorem~\ref{mainish}, based on Shirota's theorem, is possible. It would avoid us formally introducing the topology on the set of ultrafilters of $\Theta(X)$, yet we believe it would be less elementary and, of course, not self-contained. Moreover, Shirota's theorem also relies on topologisation of the ultrafilter space of a lattice so this strategy is not too different at the core. Nevertheless, we wish to present this approach too. In order to do so, we require a piece of terminology.\smallskip

\begin{definition}A distributive lattice with the least element is called an \emph{R-lattice} if it is isomorphic to a sublattice of ${\rm RO}(X)$ for a locally compact Hausdorff space $X$ whose members form a base for $X$ and have compact closures.\end{definition}
This is not the original definition of an R-lattice, however \cite[Theorem 1]{shi} asserts that it is indeed equivalent, so we take it as a definition. R-lattices have been recently studied in a broader context by Bice and Starling (\cite{bice}). The following result is a restatement of \cite[Theorem~2]{shi}.

\begin{theorem}[Shirota]\label{shirota}Let $X$ and $Y$ be locally compact Hausdorff spaces. Suppose that $A\subseteq {\rm RO}(X)$ and $B\subseteq {\rm RO}(Y)$ are R-lattices. If $A$ and $B$ are isomorphic as lattices, then $X$ and $Y$ are homeomorphic.\end{theorem}

\begin{proof}[An alternative proof of Theorem~\ref{mainish}]Let $T\colon C(X)\to C(Y)$ be a compatibility isomorphism. By Proposition~\ref{P:tau}, the map $\tau(\sigma(f)) = \sigma(Tf)$ ($f\in C(X)$) is an inclusion-preserving bijection between the lattices $$A:=\{\sigma(f)\colon f\in C(X)\}\subseteq {\rm RO}(X)\text{ and }B:=\{\sigma(f)\colon f\in C(Y)\}\subseteq {\rm RO}(Y).$$ Consequently, by Lemma~\ref{latticesaresuperb}, $\tau$ is a lattice isomorphism. As $X$ and $Y$ are compact, closures of members of $A$ and $B$ are compact, so $A$ and $B$ are R-lattices. Finally, by Theorem~\ref{shirota}, $X$ and $Y$ are homeomorphic.\end{proof}

\section{Automatic continuity of compatibility isomorphisms---Proofs of Theorems~\ref{T:ContPos} and \ref{T:Discont}}

This section is devoted to further investigations of the map $\tau$, which appeared in Proposition~\ref{P:tau}, that will finally lead to proofs of Theorems~\ref{T:ContPos} and \ref{T:Discont}.

\begin{proposition}\label{P:ClopenToClopen}
Let $X$ and $Y$ be completely regular spaces such that there exists a~compatibility isomorphism $T\colon C(X)\to C(Y)$. If $U\subseteq X$ is clopen, then $\tau(X\setminus U)=Y\setminus\tau(U)$. In particular, $\tau(U)$ is clopen. 
\end{proposition}

\begin{proof}Let us start with the easy observation that $\tau(X)=Y$. Assume not. Since we have $\tau(X)=\sigma(T\mathds{1}_X)$, we also have $\supp (T\mathds{1}_X)\neq Y$. Thus there exists a~non-zero function $g\in C(Y)$ such that $\supp g \cap \supp (T\mathds{1}_X) = \varnothing$. By applying $\tau^{-1}$ to $\sigma(g)$ and $\sigma (T\mathds{1}_X)$, one readily obtains a contradiction.\smallskip

Further, $\mathds{1}_X= \mathds{1}_U + \mathds{1}_{X\setminus U}$, and obviously $\mathds{1}_U$ and $\mathds{1}_{X\setminus U}$ are orthogonal. By Lemma~\ref{additive}, we have that $T\mathds{1}_U$ and $T\mathds{1}_{X\setminus U}$ are orthogonal as well and that $T\mathds{1}_X= T\mathds{1}_U+ T\mathds{1}_{X\setminus U}$. It follows that
$$ \{y\in Y \colon (T\mathds{1}_X)(y)\neq 0\} = \{y\in Y \colon (T\mathds{1}_U)(y)\neq 0\} \cup \{y\in Y \colon (T\mathds{1}_{X\setminus U})(y)\neq 0\}$$
with the union disjoint. By applying closure and interior operations to both sides of the last equality, we conclude that
$$\sigma(T\mathds{1}_X)= \sigma(T\mathds{1}_U) \cup \sigma(T\mathds{1}_{X\setminus U}),$$
and it is easy to see from the continuity of both functions on the right-hand side that this union must be disjoint too. This in turn implies that 
$$Y=\tau(X)=\tau(U)\cup\tau(X\setminus U).$$
As $\tau(U)$ and $\tau(X\setminus U)$ are disjoint and open, we conclude that they are in fact clopen.
\end{proof}

\begin{proposition}\label{P:EqualOnClopen}
Let $X$ and $Y$ be completely regular spaces. Suppose moreover that $T\colon C(X)\to C(Y)$ is a~compatibility isomorphism. If $U\subseteq X$ is clopen and $f,g\in C(X)$ agree on $U$, then $Tf$ and $Tg$ agree on $\tau(U)$.
\end{proposition}

\begin{proof}
We have
$$ f =f\cdot \mathds{1}_U + f\cdot \mathds{1}_{X\setminus U} \quad\text{and}\quad g =f\cdot \mathds{1}_U + g\cdot \mathds{1}_{X\setminus U},$$
and so by the hypotheses and using Lemma~\ref{additive} we conclude that
\begin{equation}\label{E:TfTgDecomp}
Tf =T(f\cdot \mathds{1}_U) + T(f\cdot \mathds{1}_{X\setminus U}) \quad\text{and}\quad Tg =T(f\cdot \mathds{1}_U) + T(g\cdot \mathds{1}_{X\setminus U}).
\end{equation}
Clearly, $\sigma(f\cdot \mathds{1}_{X\setminus U}) \subseteq  \sigma(\mathds{1}_{X\setminus U}) = X\setminus U$, whence, by Proposition~\ref{P:tau} and Proposition~\ref{P:ClopenToClopen} 
$$\sigma(T(f\cdot \mathds{1}_{X\setminus U})) \subseteq \tau(X\setminus U)= Y\setminus \tau(U).$$
Likewise,
$$\sigma(T(g\cdot \mathds{1}_{X\setminus U})) \subseteq \tau(X\setminus U)= Y\setminus \tau(U).$$
It is now clear from \eqref{E:TfTgDecomp} that $Tf$ and $Tg$ coincide on $\tau(U)$, which concludes the proof.
\end{proof}

\begin{proof}[Proof of Theorem~\ref{T:ContPos}]
We shall prove the theorem by contradiction. Let us assume that $T\colon C(X)\to C(Y)$ is a discontinuous compatibility isomorphism; this means that there exist $f\in C(X)$, a sequence $(f_n)_{n=1}^\infty$ in $ C(X)$, and $\eta>0$ such that 
\begin{equation}\label{E:claim2}
\| f_k - f\|_\infty \longrightarrow 0 \text{ as }k\to\infty\text{ and }\|Tf_n-Tf\|_\infty > \eta\quad (n\in \mathbb{N}).
\end{equation}

We \emph{claim} that there exist natural numbers $n_1<n_2<\dots$ and pairwise disjoint clopen sets $V_j\subseteq Y$ such that for each $j\in\mathbb{N}$ there exists $y_j\in V_j$ with 
\begin{equation}\label{E:claim}
|Tf_{n_j}(y_j)-Tf(y_j)|>\eta.
\end{equation}

In order to prove the claim, let us consider the equivalence relation $E$ on $Y$ determined by the decomposition of $Y$ into its connected components, and the natural quotient mapping $q\colon X\to Y/E$; we shall regard $q(x)$ simultaneously as an element of $Y/E$ and as~a subset of $Y$, and use the notation $[x]=q(x)$ (so $[x]$ is the component of $Y$ containing $x$). By \cite[Theorem 6.2.24]{engelking}, the quotient space $Y/E$ is zero-dimensional (and compact). Let us observe that, in addition to that, $Y/E$ is also perfect. Indeed, if $[x]\in Y/E$ is isolated, then $\{[x]\}$ is a clopen set, and thus so must be $[x]$ when regarded as a subset of $Y$; in particular, $[x]$ has non-empty interior, contrary to our assumption. \smallskip

From \eqref{E:claim2} we infer that we may choose a sequence $(\tilde{y}_j)_{j=1}^\infty$ in $Y$ such that 
\begin{equation}\label{E:claim3}
|Tf_j(\tilde{y}_j) - Tf(\tilde{y}_j)|>\eta\quad (j\in\mathbb{N}).
\end{equation}
Since $X$ is homeomorphic to $Y$ as asserted by Theorem~\ref{mainish}, our assumptions imply that both spaces are sequentially compact, so $([\tilde{y}_j])_{j=1}^\infty$ has an accumulation point $[\tilde{y}]$, say. Now, consider $[\tilde{y}]$ as a subset of $Y$. It is possible that $\tilde{y}_1\in [\tilde{y}]$; in fact, we might even have $[\tilde{y}_j] = [\tilde{y}]$ for all $j$. Because of this we need to choose a further auxiliary sequence $(\tilde{z}_j)_{j=1}^\infty$.

The recursive construction of this sequence begins by choosing any point  $\tilde{z}_1\in Y\setminus [\tilde{y}]$ and a clopen set $\V_1\subseteq Y/E$ such that $|Tf_1(\tilde{z}_1)-Tf(\tilde{z}_1)|>\eta$, $[\tilde{z}_1]\in\V_1$ and $[\tilde{y}]\notin\V_1$; this can be done because $[\tilde{y}]$ is nowhere dense by the assumption of the theorem, $Tf$ and $Tf_1$ are continuous, and $Y/E$ is zero-dimensional. Set $n_1=1$. Next, assume we have chosen natural numbers $n_1<n_2<\ldots<n_k$, points $\tilde{z}_1,\ldots, \tilde{z}_k\in Y$ and pairwise disjoint clopen sets $\V_1,\ldots, \V_k\subseteq Y/E$ such that for each $j=1,\ldots,k$ we have $[\tilde{z}_j ]\in\V_j$ and $|Tf_{n_j}(\tilde{z}_j)-Tf(\tilde{z}_j)|>\eta$, and $[\tilde{y}]\notin \bigcup_{j=1}^k \V_j$. Since $[\tilde{y}]$ is an accumulation point of $([\tilde{y}_j])_{j=1}^\infty$, we can find a natural number $n_{k+1}>n_k$ such that $[\tilde{y}_{n_{k+1}}]\notin \bigcup_{j=1}^k \V_j$. Now, using the fact that $Y/E$ is perfect, we find the next point $\tilde{z}_{k+1}$ so close to $\tilde{y}_{n_{k+1}}$ that $|Tf_{n_{k+1}}(\tilde{z}_{k+1})-Tf(\tilde{z}_{k+1})|>\eta$ and so that $[\tilde{z}_{k+1}]\notin\bigcup_{j=1}^k\V_j$ and $[\tilde{z}_{k+1}]\neq[\tilde{y}]$---similarly as in the first step, we use \eqref{E:claim3}, the continuity of $Tf$ and $Tf_{n_{k+1}}$, and the nowhere denseness of $[\tilde{y}_{n_k+1}]$ to achieve that. We also pick a clopen set $\V_{k+1}$ disjoint from all $\V_j$, $j=1,\ldots,k$, and such that $[\tilde{z}_{k+1}]\in \V_{k+1}$ and $[\tilde{y}]\notin \V_{k+1}$. Finally, having finished the recursive construction, we define $y_j=\tilde{z}_j$ and $V_j=q^{-1}(\V_j)$, concluding the proof of the claim. \smallskip

In order to obtain the desired contradiction, we now aim to construct  $\tilde{f}\in C(X)$ which cannot be mapped by $T$ to a continuous function on $Y$. First we need to consider the mapping $\tau$ from Proposition~\ref{P:tau}; $\tau$ is an inclusion-preserving bijection between $\{\sigma(f)\colon f\in C(X)\}$ and $\{\sigma(g)\colon g\in C(Y)\}$ whose inverse also preserves inclusion. Observe that clopen sets in $Y$ are of the form $\sigma(g)$ because the corresponding characteristic functions are continuous. Hence, we may set $U_j:=\tau^{-1}(V_j)$ ($j\in\mathbb{N}$), and we know from Proposition~\ref{P:ClopenToClopen} that the sets $U_j$ are clopen. Moreover, by passing to a subsequence, we can guarantee that the sequence $(y_j)_{j=1}^\infty$ converges to a point $z\in Y$.

We define the function $\tilde{f}$ as follows:
$$\tilde{f}(x)=
\begin{cases}
f_{n_{2j}}(x), & x\in U_{2j}, \\
f(x), & x\in X\setminus \bigcup_{j=1}^\infty U_{2j}.
\end{cases}
$$
Let us observe that $\tilde{f} \in C(X)$; this is equivalent to showing that $h:=\tilde{f}-f$ is continuous. To this end, let us take any open set $G$ in the scalar field. We shall distinguish two cases: if $0\notin G$, then $h^{-1}(G)\subseteq \bigcup_{j=1}^\infty U_{2j}$, so we obtain that
$$h^{-1}(G)=\bigcup_{j=1}^\infty \left(h^{-1}(G)\cap U_{2j} \right) = \bigcup_{j=1}^\infty\left( (f_{n_{2j}}-f)^{-1}(G)\cap U_{2j} \right),$$
which is an open set. On the other hand, if $0\in G$, then we may pick $\varepsilon>0$ and $j_0\in\mathbb{N}$ such that $U(0,\varepsilon)\subseteq G$ and for all $j\geqslant j_0$, $\|f_{n_{2j}}-f\|<\varepsilon$. That is, for all $j\geqslant j_0$ we have
$$ (f_{n_{2j}}-f)^{-1}(G)\supseteq (f_{n_{2j}}-f)^{-1}\big(U(0,\varepsilon)\big) = X.$$
The following computations demonstrate that $h^{-1}(G)$ is open:
\begin{equation*}
\begin{aligned}
h^{-1}(G) & = \left(h^{-1}(G)\cap \Big(X\setminus\bigcup_{j=1}^\infty U_{2j} \Big) \right) \cup \left(h^{-1}(G)\cap\bigcup_{j=1}^\infty U_{2j}  \right) \\
& = \left( X\setminus\bigcup_{j=1}^\infty U_{2j} \right) \cup \bigcup_{j=1}^\infty \left( (f_{n_{2j}}-f)^{-1}(G)\cap U_{2j} \right) \\
& = \left( X\setminus\bigcup_{j=1}^\infty U_{2j} \right) \cup \bigcup_{j=j_0}^\infty U_{2j} \cup \bigcup_{j=1}^{j_0-1} \left( (f_{n_{2j}}-f)^{-1}(G)\cap U_{2j} \right) \\
& =  \left( X\setminus\bigcup_{j=1}^{j_0-1} U_{2j} \right) \cup \bigcup_{j=1}^{j_0-1} \left( (f_{n_{2j}}-f)^{-1}(G)\cap U_{2j} \right).
\end{aligned}
\end{equation*}

As $\tilde{f}$ is continuous, $T$ is defined at $\tilde{f}$. We may now turn our attention to the (\emph{a~priori} continuous) function $T\tilde{f}$. By Proposition~\ref{P:EqualOnClopen}, for each $j\in\mathbb{N}$, 
$$T\tilde{f}=Tf\text{ in }V_{2j-1}=\tau(U_{2j-1}) \quad\text{and}\quad T\tilde{f}=Tf_{n_{2j}}\text{ in }V_{2j}=\tau(U_{2j});$$  
since $y_j\in V_j$, we now have for each $j$,
$$T\tilde{f}(y_{2j-1})=Tf(y_{2j-1}) \quad\text{and}\quad  T\tilde{f}(y_{2j})=Tf_{n_2j}(y_{2j}).$$
It follows from the above considerations, the fact that $y_j\longrightarrow z$ and continuity of $Tf$ that
$$\lim_{j\to\infty} T\tilde{f}(y_{2j-1}) = \lim_{j\to\infty} Tf(y_{2j-1}) = Tf(z).$$
On the other hand, by \eqref{E:claim}, 
$$\lim_{j\to\infty} T\tilde{f}(y_{2j}) = \lim_{j\to\infty} Tf_{n_{2j}}(y_{2j}) \neq \lim_{j\to\infty} Tf(y_{2j}) =  Tf(z).$$
This contradicts the continuity of $T\tilde{f}$, and the proof is complete.
\end{proof}

The next simple lemma will be useful in the proof of Theorem~\ref{T:Discont}.

\begin{lemma}\label{L:connprec}
Let $X$ be a compact Hausdorff space, $F\subseteq X$ be connected, $f,g\in C(X)$ satisfy $g\preceq f$, and let $f(x)\neq 0$ for each $x\in F$. Then either $g(x)=f(x)$ for each $x\in F$ or $g(x)=0$ in $F$.
\end{lemma}

\begin{proof}
Set $H:=\{x\in F\colon f(x)=g(x) \}$; then $H$ is closed in $F$ as both functions are continuous. On the other hand, from the definition of compatibility ordering it follows that for each $x\in X$, either $g(x)=0$ or $g(x)=f(x)$. Hence, $H=F\setminus K$ where $$K=\{x\in F\colon g(x)=0\}.$$ Since $K$ is closed in $F$, $H$ must be open in $F$. So $H$ is in fact clopen in $F$, and from the connectedness of $F$ it follows that either $H=F$ or $H=\varnothing$.
\end{proof}

\begin{proof}[Proof of Theorem~\ref{T:Discont}]
Let $F$ be a connected component of $X$ with non-empty interior, set $\A = \{f\in C(X) \colon f(x)\neq 0 \text{ for all }x\in F \}$, $\tilde{\A}=\{f|_F\colon f\in\A\}$ and let us pick any bijection $\tilde{\varphi}\colon \tilde{\A}\to \tilde{\A}$ which is not continuous (with respect to the topology at question) and satisfies that for any $\tilde{f}\in \tilde{\A}$, 
\begin{equation}\label{E:boundary}
(\tilde{\varphi}\tilde{f
})|_{\partial F}=\tilde{f}|_{\partial F}
\end{equation}
(of course, $\partial F$, the boundary of $F$, is taken with respect to $X$).
One can readily see that such bijections exist in abundance thanks to the interior of $F$ being non-empty. To give a very simple example, consider a non-negative continuous bump function $\tilde{g}\in C(F)$ such that $\supp \tilde{g}\cap \partial F = \varnothing$, and define $\tilde{f}_1$ to be the constant $1$ function on $F$ and $\tilde{f}_2=\tilde{f}_1+\tilde{g}$. Then $\tilde{f}_1, \tilde{f}_2 \in \tilde{\A}$ and we may define a bijection $\tilde{\varphi}\colon\tilde{\A}\to\tilde{\A}$ by simply switching $\tilde{f}_1$ and $\tilde{f}_2$, that is, $\tilde{\varphi}(\tilde{f}_1)=\tilde{f}_2$, $\tilde{\varphi}(\tilde{f}_2)=\tilde{f}_1$ and $\tilde{\varphi} (\tilde{f})=\tilde{f}$ for any $\tilde{f}\in \tilde{\A}\setminus\{\tilde{f}_1, \tilde{f}_2\}$. Clearly, $\tilde{\varphi}$ is a~bijection satisfying condition \eqref{E:boundary}, and if $\tilde{g}\neq 0$, then it is clearly not continuous (in either of the two topologies). \smallskip

Having chosen any $\tilde{\varphi}\colon\tilde{\A}\to\tilde{\A}$ as described above, we can now extend its domain to the whole of $C(F)$ using simply the identity, so $\tilde{\varphi}(\tilde{f})=\tilde{f}$ for any $\tilde{f}\in C(F)\setminus \tilde{\A}$; finally, we define $\varphi\colon C(X)\to C(X)$ as follows.
\begin{equation*}
(\varphi f)(x):=
\begin{cases}
f(x)\quad & \text{if }x\in X\setminus F,\\
(\tilde{\varphi}(f|_F))(x) \quad & \text{if } x\in F.
\end{cases}
\end{equation*}
Condition \eqref{E:boundary} and the definitions of $\tilde{\varphi}$ and $\varphi$ make it clear that for any $f\in C(X)$ it is true that $\varphi(f)\in C(X)$. Furthermore, we know that $\tilde{\varphi}$ is a discontinuous bijection, which easily makes $\varphi\colon C(X)\to C(X)$ a discontinuous bijection, and we are left to show that $\varphi$ is also a compatibility isomorphism.\smallskip

To that end, let us first observe that $\varphi(f)=f$ for all $f\in C(X)\setminus\A$; it is also plain that $(\varphi f)|_{X\setminus F}= f|_{X\setminus F}$ even for $f\in\A$. Now, pick any $f,g\in C(X)$ with $f\preceq g$. If $g\notin \A$, then $f\notin \A$, whence $\varphi(f)=f\preceq g=\varphi(g)$. If $f,g\in \A$, then (due to Lemma~\ref{L:connprec}) $f|_F=g|_F$, whence $(\varphi f)|_F = (\varphi g)|_F$ by the definition of $\varphi$, and above we have observed $(\varphi f)|_{X\setminus F}= f|_{X\setminus F}$ and the same for $g$; it follows from these three facts that $\varphi(f)\preceq\varphi(g)$. Finally, if $g\in \A$ and $f\notin \A$, then Lemma~\ref{L:connprec} yields that $f(x)=0$ for all $x\in F$ and (as $f\notin \A$) we know the same to be true for $\varphi(f)=f$. As $\varphi$ leaves both $f$ and $g$ unchanged in the complement of $F$, it is again obvious that $\varphi(f)\preceq\varphi(g)$. This completes the discussion of possible cases and the proof.
\end{proof}

\end{document}